\documentclass[a4paper, reqno, 12pt]{amsart}

\usepackage[utf8]{inputenc}
\usepackage{amsthm,amsfonts,amssymb,amsmath,amsxtra,amsrefs}
\usepackage[new]{old-arrows}
\usepackage{mathtools}
 \usepackage{bbold}
\usepackage{latexsym}
 \usepackage{stmaryrd}
\usepackage{graphicx}
\usepackage{tikz}
\usepackage{tikz-cd}
\usepackage{todonotes,cancel}
\usepackage[margin=1.25in]{geometry}
\usepackage{mathrsfs}
\usepackage{bm}

\usepackage[all]{xy}
\SelectTips{cm}{}
\usepackage{xr-hyper}
\usepackage[colorlinks=false,
   citecolor=Black,
   linkcolor=Red,
   urlcolor=Blue]{hyperref}
\usepackage{verbatim}
\RequirePackage{xspace}
\RequirePackage{etoolbox}
\RequirePackage{varwidth}
\RequirePackage{enumitem}
\RequirePackage{tensor}
\RequirePackage{mathtools}
\RequirePackage{longtable}
\RequirePackage{multirow}

\usepackage{rotating}

\tikzset{
    labl/.style={anchor=south, rotate=90, inner sep=.5mm}
}

\newtheorem{thm}{Theorem}
\newtheorem{theorem}[thm]{Theorem}
\newtheorem{thmintro}{Theorem}

\newtheorem{propintro}[thmintro]{Proposition}

\newtheorem{prop}[thm]{Proposition}
\newtheorem{proposition}[thm]{Proposition}

\newtheorem{lemma}[thm]{Lemma}

\newtheorem{cor}[thm]{Corollary}
\newtheorem{corollary}[thm]{Corollary}

\theoremstyle{definition}
\newtheorem{definition}[thm]{Definition}

\newtheorem{remark}[thm]{Remark}

\numberwithin{equation}{section}
\numberwithin{thm}{section}

\newcommand{\BA}{\ensuremath{\mathbb {A}}\xspace}

\newcommand{\BP}{\ensuremath{\mathbb {P}}\xspace}

\newcommand{\BZ}{\ensuremath{\mathbb {Z}}\xspace}

\newcommand{\CI}{\ensuremath{\mathcal {I}}\xspace}

\newcommand{\CL}{\ensuremath{\mathcal {L}}\xspace}

\newcommand{\CO}{\ensuremath{\mathcal {O}}\xspace}

\newcommand{\X}{\ensuremath{\mathbf {X}}\xspace}
\newcommand{\Y}{\ensuremath{\mathbf {Y}}\xspace}

\newcommand{\Ir}{\mathbf{I}_\circ'}
\newcommand{\I}{\mathbf{I}}

\newcommand{\wt}{\text{wt}}

\begin{document}

\title{Double Bruhat cells in symmetric spaces}

\author[Huanchen Bao]{Huanchen Bao}
\address{Department of Mathematics, National University of Singapore, Singapore.}
\email{huanchen@nus.edu.sg}

\author[Jinfeng Song]{Jinfeng Song}
\address{Department of Mathematics, The Hong Kong University of Science and Technology, Clear Water Bay, Hong Kong SAR.}
\email{jfsong@ust.hk}

\begin{abstract}
This paper studies the closures of the
double Bruhat cells of symmetric spacess in its equivariant embeddings. We call them large double Bruhat varieties. Our aim is to establish the double
Bruhat stratification as a natural setting for extending the geometry of
Richardson varieties to symmetric spaces.

For the wonderful compactification of the associated adjoint symmetric space, we show their scheme-theoretic intersections with the unique closed orbit are reduced unions of projected Richardson varieties. This connection with Richardson geometry on partial flag varieties supports large double Bruhat varieties as a natural class of subvarieties in symmetric-space embeddings.

In positive characteristic, we prove that every equivariant embedding admits a Frobenius splitting compatible with all large double Bruhat varieties and all boundary divisors. These varieties are strongly \(F\)-regular, and globally \(F\)-regular when the ambient embedding is projective. In characteristic zero, they are of strongly \(F\)-regular type, and of globally \(F\)-regular type in the projective case; in particular, they are normal, Cohen–Macaulay, and have rational singularities.
\end{abstract}

\maketitle


\setcounter{section}{-1}

\section{Introduction}\label{sec:introduction}

\subsection{Background and motivation}
Let $G$ be a connected reductive group over an algebraically closed field
$k$ of characteristic different from $2$, let $\theta$ be an involution,
and put $K=G^\theta$ and $S=G/K$. This paper studies the closures of the
double Bruhat cells of $S$ in its equivariant embeddings. Our aim is to establish the double
Bruhat stratification as a natural setting for extending the geometry of
Schubert and Richardson varieties to symmetric spaces.

Frobenius splittings of symmetric varieties are already known.
Brion--Inamdar proved that reductions of spherical varieties modulo
sufficiently large primes admit splittings compatible with all
$G$-orbit closures \cite{BI94}*{Theorem~1}.
For wonderful compactifications of symmetric spaces,
De Concini--Springer obtained such splittings in every odd characteristic
\cite{DCS99}*{Theorem~5.9}; Tange extended the construction to arbitrary
equivariant embeddings \cite{Tan18}*{Theorem~2.2 and Section~3}.
The compatibility assertions in these results concern $G$-orbit closures.
They do not provide compatibility with a finer stratification inside the
open symmetric space.

For group embeddings, substantially more is known. In the wonderful
compactification of a semisimple adjoint group, Brion--Polo constructed
a splitting compatible with all large Schubert varieties, the closures
of the double cosets $B\dot wB$ \cite{BP00}*{Theorem~2}.
Brion--Thomsen extended these splitting results to general group
embeddings and proved that projective large Schubert varieties are
globally $F$-regular \cite{BT06}*{Proposition~4.2 and Theorem~4.3}.
He--Thomsen proved that all $B\times B$-orbit closures are strongly
$F$-regular in arbitrary group embeddings and globally $F$-regular
in projective embeddings \cite{HT07}*{Theorem~8.1 and Corollary~8.4}.
These properties imply normality and Cohen--Macaulayness, and global
$F$-regularity also gives vanishing of higher cohomology for nef line
bundles \cite{Smi00}*{Proposition~4.1 and Corollary~4.3}.
They suggest seeking analogous results for subvarieties in embeddings
of symmetric spaces.

Earlier extensions in this direction concern $B$-orbit closures in
$G/K$, or equivalently $K$-orbit closures in $G/B$; see
\cite{Bri01,HT12} and the quantum symmetric pair constructions of
Bao--Song \cite{BS22}*{Theorems~6.4 and~6.8}.
Their geometry is more complicated: $B$-orbit closures can be nonnormal,
and there need not be a splitting compatible with all of them, since
their scheme-theoretic intersections can be nonreduced
\cite{Bri01}*{Introduction, p.~266}; see also
\cite{Tan18}*{Introduction}.
Thus the favourable properties of group orbit closures do not extend
uniformly to this stratification.

The double Bruhat stratification of $S = G/K$ has its origins in Evens--Lu's study
of Poisson geometry \cite{EL06}*{Section~4, especially Theorem~4.14}.
The first author's study of these strata in total positivity on symmetric spaces
\cite{Bao26}*{Section~1.3 and Definition~2.2} suggests that they provide
the appropriate stratification for 
symmetric spaces. 

Choose the opposite Borels $B^+,B^-$ and the
$\theta$-stable maximal torus as in Section~\ref{sec:preliminaries}.
The embedding $\iota$ and the double Bruhat strata are given by
\[
\iota:S\hookrightarrow G,\quad gK\longmapsto g\theta(g)^{-1},
\qquad
\mathring S_u=\iota^{-1}(B^+\dot uB^-\dot w_0\dot w_\bullet).
\]
Here $w_0$ is the longest Weyl group element and $w_\bullet$ is the
longest element for the black nodes of the Satake diagram.
The nonempty strata $\mathring S_u$ are indexed by
$u\leq w^\bullet=w_0w_\bullet$. 
For the symmetric space $(G\times G)/\operatorname{diag}G$, they recover
the usual double Bruhat cells of $G$. Set
$S_u$ to be the closure of $\mathring{S}_u$ in $S$.

\subsection{Main results} We now explain the main results of this paper.

\begin{thmintro}[Theorem~\ref{thm:main} and Proposition~\ref{prop:DBFregular}]\label{thm:intro-affine} Assume that $\operatorname{char}k=p>2$. 
There is a Frobenius splitting of $S$ compatible with every $S_u$.
For every $u\leq w^\bullet$, the coordinate ring $k[S_u]$ is strongly
$F$-regular.
\end{thmintro}

For an equivariant embedding $X$ of $S$, let $X_u$ denote the closure
of $\mathring S_u$ in $X$, and let $\partial X=X\setminus S$ be the boundary.
We call $X_u$ a \emph{large double Bruhat variety}.
By an equivariant embedding we always mean a normal $G$-variety which contains $S$ as an open orbit.

\begin{thmintro}[Theorem~\ref{thm:main}, Corollary~\ref{cor:strong-f-regular-strata}, and Theorem~\ref{thm:global-f-regular-strata}]\label{thm:intro-embeddings} Assume that $\operatorname{char}k=p>2$. 
Let $X$ be any equivariant embedding of $S$.
\begin{enumerate}[label=\textup{(\arabic*)},leftmargin=*]
\item $X$ admits a Frobenius splitting compatible with every boundary
      divisor and every $X_u$.
\item Every $X_u$ is strongly $F$-regular.
\item If $X$ is projective, every $X_u$ is globally $F$-regular.
\end{enumerate}
\end{thmintro}

In characteristic zero, we obtain the following consequences.

\begin{thmintro}[Theorem~\ref{thm:characteristic-zero-f-regular} and Corollary~\ref{cor:characteristic-zero-singularities}]\label{thm:intro-zero}
Assume that $\operatorname{char}k=0$ and let $X$ be an equivariant
embedding of $S$. Every $X_u$ is of strongly $F$-regular type; in
particular, it is normal, Cohen--Macaulay, and has rational singularities.
If $X$ is projective, every $X_u$ is of globally $F$-regular type.
\end{thmintro}

For group embeddings, some special cases of these subvarieties $X_u$ are often referred to as ``large Schubert
varieties''; see \cite{BP00} and \cite{BT06}. We use the term ``large double Bruhat varieties” to reflect their definition as closures of double Bruhat cells. Proposition \ref{Prop:D} connects these varieties with projected Richardson geometry on the closed orbit, instead of the Schubert geometry.

\begin{propintro}[Corollary~\ref{cor:si}]\label{prop:intro-closed-orbit}\label{Prop:D}
Assume that $\operatorname{char}k\neq2$. Let $\X$ be the wonderful
compactification of the associated adjoint symmetric space
$S_{\mathrm{ad}}$, with closed orbit $Y\cong G/P_{\I_\bullet}$.
For every $u\leq w^\bullet$, the scheme-theoretic intersection
$\X_u\cap Y$ is a reduced union of projected Richardson varieties.
Moreover, $\bigcup_{i\in\I}(\X_{s_i}\cap Y)$ is the union of all
projected Richardson divisors of $Y$.
\end{propintro}

Here $\I_\bullet \subset \I$ and the parabolic subgroup $P_{I_\bullet}$ are determined by the corresponding Satake diagram. 

\subsection{Organization of the paper}
Section~\ref{sec:preliminaries} recalls preliminaries.
Section~\ref{sec:enveloping-variety} treats enveloping varieties and line
bundles on the wonderful compactification.
Section~\ref{sec:wonderful-splitting} constructs its Frobenius splitting.
Section~\ref{sec:equivariant-embeddings} establishes the results for general
equivariant embeddings, including the characteristic-zero consequences.

\subsection*{Acknowledgements} The authors used ChatGPT to assist with proof exploration and the typesetting of this paper. The authors take full responsibility for the content of the paper and the correctness of its results. HB is supported by the MOE grant A-8003582-00-00. JS is supported by the Glorious Sun Charity Fund.

\section{Preliminaries}\label{sec:preliminaries}

\subsection{Reductive groups and symmetric pairs}\label{subsec:prelim-symmetric-pairs}

Throughout the paper, unless otherwise specified, $k$ is an algebraically closed field. Let $G$ be a connected reductive group over $k$, and choose a maximal torus $T\subset G$ contained in a Borel subgroup $B^+$. Let $B^-$ be the opposite Borel subgroup with respect to $T$. Let $U^\pm$
be the unipotent radicals of $B^\pm$, respectively. We denote by $X=\operatorname{Hom}(T,k^\times)$ and $Y=\operatorname{Hom}(k^\times,T)$ the character and cocharacter lattices, with perfect pairing $\langle\ ,\ \rangle:Y\times X\to \BZ$. Let $\{\alpha_i\}_{i\in \I}$ and $\{\alpha_i^\vee\}_{i\in \I}$ be the corresponding simple roots and simple coroots. Let $W=\langle s_i\mid i\in \I\rangle$ be the Weyl group, with longest element $w_0$. We write
\[
X^+=\{\lambda\in X\mid \langle \alpha_i^\vee,\lambda\rangle\geq 0\text{ for all }i\in \I\}
\]
for the set of dominant weights, and we use the usual dominance order: $\lambda\leq \mu$ if $\mu-\lambda$ is a non-negative integral linear combination of simple roots. 

Let $\theta$ be an algebraic group involution of $G$, and let $K=G^\theta$ be its fixed point subgroup. We call $K$ a \emph{symmetric subgroup} and $G/K$ a \emph{symmetric space}. A torus $A\subset G$ is called \emph{$\theta$-split} if $\theta(a)=a^{-1}$ for all $a\in A$. We choose $T$ to be $\theta$-stable and to contain a maximal $\theta$-split torus. We also choose $B^+$ containing $T$ such that $\theta(B^+)\cap B^+$ has minimal dimension. With this choice, the involution $\theta$ induces involutions $\theta_X$ and $\theta_Y$ on $X$ and $Y$. There is a subset $\I_\bullet\subset \I$ and an involution $\tau:\I\to \I$ preserving $\I_\bullet$ such that, if $w_\bullet$ denotes the longest element of the parabolic subgroup $W_{\I_\bullet}$, then
\[
w_\bullet(\alpha_i)=-\alpha_{\tau i}\quad (i\in \I_\bullet),
\qquad
\theta_X(\alpha_i)=-w_\bullet\alpha_{\tau i}\quad (i\in \I).
\]
We write $\I_\circ=\I\setminus \I_\bullet$. The datum $(\I=\I_\bullet\sqcup\I_\circ,\tau,Y,X,\langle\ ,\ \rangle,\theta_X,\theta_Y)$ is the \emph{$\imath$-root datum} associated with $(G,\theta)$. This is the convention used in Bao--Song \cite{BS25b}*{Section 2.1} and Song \cite{Son25}*{Section 3.3}.

Put $w^\bullet=w_0w_\bullet$, $P = P^+_{\I_\bullet} \supset B^+$, $U_{P}^\pm=R_u(P^\pm_{\I_\bullet})$. Here $P^\pm_{\I_\bullet}$ are the standard  parabolic subgroups associated with $\I_\bullet$. Write $i\mapsto i^*$ for
the involution of $\I$, such that $w_0s_iw_0=s_{i^*}$, and set $a^*=w_0aw_0$ for
$a\in W$. The Satake conventions give
\[
\begin{gathered}
w_0w_\bullet=w_\bullet w_0,\qquad (w^\bullet)^{-1}=w^\bullet,\qquad\theta(w^\bullet)=w^\bullet,\\
\theta(B^\pm)=\dot w^\bullet B^\pm(\dot w^\bullet)^{-1},\qquad
\theta(U_P^\pm)=U_P^\mp;
\end{gathered}
\]
see \cite{Bao26}*{Section~2.1}. Here representatives in $N_G(T)$ are fixed;
equalities involving their products are understood modulo $T$.

The \emph{spherical weight lattice} is
\[
\breve X=\{\lambda-\theta_X(\lambda)\mid \lambda\in X\}\subset X,
\]
and we write $\overline{\lambda}=\lambda-\theta_X(\lambda)$ for $\lambda\in X$. We also set $\breve X^+=\breve X\cap X^+$. This notation follows the convention for symmetric spaces and their embeddings \cite{BS25b}*{Section 2.1}.

Throughout this paper, we assume that $G/K$ is \emph{nondegenerate}:
the induced involution $\theta_{\mathrm{ad}}$ on $G/Z(G)$ is nontrivial,
or equivalently $\I_\bullet\ne\I$. This excludes the case where
$G/K\cong T/T^\theta$ is a torus, including a point. In the excluded case,
every equivariant embedding is toric, the only nonempty double Bruhat
stratum is $G/K$ itself, and the adjoint wonderful compactification is
a point. The results of this paper then reduce to standard results for
toric varieties or to statements about a point. For Frobenius splitting,
see \cite{BK05}*{Exercise~1.3.E.6 and Theorem~1.4.10}; for global
$F$-regularity and its characteristic-zero analogue, see
\cite{Smi00}*{Proposition~6.4}, with the singularity and vanishing
consequences in \cite{Smi00}*{Proposition~4.1 and Corollaries~4.3 and~5.3}.

\subsection{Modules and filtered coordinate rings}\label{subsec:prelim-based-modules}

For $\lambda\in X^+$, let $V(\lambda)$ be the Weyl module of $G$ of highest weight $\lambda$, with highest weight vector $v_\lambda$. Its linear dual $V(\lambda)^*$ is a dual Weyl module. We write $v_{w\lambda} \in V(\lambda)$ for a chosen extremal weight vector, for $w\in W$.

Recall the filtered subspaces $k[G]_{\leq\mu}\subset k[G]$ and
\[
k[G/K]_{\leq\mu}=k[G]_{\leq\mu}\cap k[G/K],
\qquad \mu\in X,
\]
from \cite{BS24}*{Sections~2.1.4 and~4.4}. The Peter--Weyl filtration of $k[G]$ is a good filtration with sections $V(\lambda)^*\otimes V(-w_0\lambda)^*$, and the nonzero graded pieces of the filtration on $k[G/K]$ are the dual Weyl modules $V(\lambda)^*$ for spherical dominant weights $\lambda$. These filtered spaces will be used in the construction of the enveloping variety and the distinguished sections on the wonderful compactification.

\subsection{The Lang map and symmetric-space strata}\label{subsec:prelim-symmetric-strata}

Define the map $\tilde{\iota}:G\to G$ by $\tilde\iota(g)=g\theta(g)^{-1}$. It is clear that $\tilde\iota$ is constant on right $K$-cosets. Therefore it descends to a morphism
\[
\iota:G/K\longrightarrow G.
\]
The twisted $G$-action on $G$ is given by $g*x=gx\theta(g)^{-1}$ for $g\in G$ and $x\in G$. Let $S_\theta=\{g\theta(g)^{-1}\mid g\in G\}\subset G$ be the twisted orbit of $e\in G$. By Springer \cite{Sp85}*{Proposition 2.2}, $S_\theta$ is a closed subvariety of $G$, and the induced map
\[
G/K\longrightarrow S_\theta,
\qquad
 gK\longmapsto g\theta(g)^{-1}
\]
is an isomorphism of affine $G$-varieties. The $B^\pm$-orbits on $G/K$ may equivalently be viewed as the corresponding twisted $B^\pm$-orbits on $S_\theta$. By Springer \cite{Sp85} and Richardson--Springer \cite{RS90}, there are only finitely many such orbits.

Let $\breve T=T/T^\theta$ be the quotient torus with character lattice
$\breve X$, let $\breve Y$ be its cocharacter lattice, and put
$r=\dim\breve T$. Denote by $S=G/K$ the symmetric space.
With $U_P^-=R_u(P_{\I_{\bullet}}^-)$ as in Section~\ref{subsec:prelim-symmetric-pairs},
the local structure theorem \cite{DCP83}*{Proposition 2.3}
gives an isomorphism by left multiplication
\[
U_P^-\times \breve T\cong B^-K/K.
\]
In particular,
\[
\dim S=\ell(w^\bullet)+r,
\]
with the notation fixed above.

We next recall the stratification of $S$ motivated by Poisson geometry. For $u\in W$, define
\[
\mathring S_u
=
\iota^{-1}(B^+\dot u B^-\dot w_0\dot w_\bullet)
=
\iota^{-1}(B^+\dot u\dot w_0\dot w_\bullet\theta(B^-)).
\]
We write $S_u=\overline{\mathring S_u}$ for the Zariski closure.

\begin{definition}
We call the stratum $\mathring S_u$ the \emph{(open) double Bruhat cell} for the symmetric space $G/K$.
\end{definition}

In the group case, equivalently for the symmetric space $(G\times G)/\Delta(G)$, this construction recovers the usual double Bruhat cells of $G$. This motivates the name.

\begin{remark}
Let $\widetilde G=G\times G$. We consider the involution
\[
\widetilde\theta:G\times G\to G\times G,
\qquad
(g,h)\longmapsto(\theta(h),\theta(g)).
\]
Denote the fixed-point group by $\widetilde K=\widetilde G^{\widetilde\theta}$. Consider the symmetric space $\widetilde G/\widetilde K$. It follows from Evens--Lu \cite{EL06} that the intersections of diagonal $G$-orbits with $(B^+,B^-)$-orbits are precisely the $T$-leaves of the natural Poisson structure on $\widetilde G/\widetilde K$ as a Poisson homogeneous space; see also Lu \cite{Lu14}. The map
\[
G\times G\to G,
\qquad
(g,h)\longmapsto g\theta(h)^{-1}
\]
induces an isomorphism $\widetilde G/\widetilde K\cong G$. Under this isomorphism, diagonal $G$-orbits in $\widetilde G/\widetilde K$ are identified with the twisted conjugacy classes, and $(B^+,B^-)$-orbits are identified with $(B^+,\theta(B^-))$ double cosets in $G$. Under this isomorphism, $S$ is precisely the diagonal $G$-orbit of $\widetilde K$. This leads to the stratum $\mathring S_u$.
\end{remark}

\begin{lemma}\label{lem:prelim-strata-intersection}
\begin{enumerate}
\item $\mathring S_u\neq\emptyset$ if and only if $u\leq w^\bullet$ in the Bruhat order.
\item $S_u=\bigsqcup_{u\leq v\leq w^\bullet}\mathring S_v$.
\item For $u\leq w^\bullet$, $\mathring S_u\subset S$ is a smooth, connected, affine subvariety of dimension $\ell(w^\bullet)-\ell(u)+r$.
\end{enumerate}
\end{lemma}

\begin{proof}
We prove Part~(1). Let $\theta'=\operatorname{Int}(\dot w^\bullet)\theta$. Then $\theta'$ fixes $B^+$ and $B^-$. Let $C$ be the $\theta'$-twisted conjugacy class of $(\dot w^\bullet)^{-1}$. Then $\iota(G/K)=C\cdot\dot w^\bullet$. Since $B^+K/K\subset G/K$ is dense, $B^+(\dot w^\bullet)^{-1}B^+\cap C$ is dense in $C$. By \cite{CLT10}*{Section 2.4}, the intersection 
\[
C\cap B^+\dot uB^-=(\iota(G/K)\cap B^+\dot uB^-\dot w^\bullet)(\dot w^\bullet)^{-1}
\]
is non-empty if and only if $u\leq w^\bullet=(w^\bullet)^{-1}$.

We prove Part~(2). We follow the notation in the preceding remark. By the Bruhat decomposition, we have
\[
\overline{B^+\dot uB^-\dot w^\bullet}
=
\bigsqcup_{u\leq v}B^+\dot vB^-\dot w^\bullet.
\]
We can equivalently consider the closure relations in $\widetilde G/\widetilde K$, since $S$ is closed in $G$. Then we can apply Richardson's theorem. In particular,
\[
S_u
=
S\cap \iota^{-1}(\overline{B^+\dot uB^-\dot w^\bullet})
=
\bigsqcup_{u\leq v\leq w^\bullet}\mathring S_v.
\]
The finiteness condition in Richardson's theorem is not needed here, since $S$ is closed; see \cite{Ri92}*{Theorem 2.1}. The assertions about smoothness, connectedness, and dimension in Part~(3) follow from \cite{Ri92}*{Corollary 1.5}. Finally, $B^+\dot uB^-\dot w^\bullet$ is affine, and $\iota(\mathring S_u)$ is closed in this double Bruhat cell. Hence $\mathring S_u$ is affine.
\end{proof}

\subsection{Frobenius splittings}\label{subsec:prelim-frobenius}
We follow \cite{BK05}*{Definitions 1.1.2--1.1.3} in this subsection.  
Assume throughout this subsection that $\operatorname{char}k=p>0$. Let $X$ be a separated scheme of finite type over $k$, and let $F:X\to X$ be the absolute Frobenius morphism. A Frobenius splitting of $X$ is an $\CO_X$-linear map $\varphi:F_*\CO_X\to\CO_X$ such that $\varphi\circ F^\#=\operatorname{id}_{\CO_X}$. If $Y\subset X$ is a closed subscheme with ideal sheaf $\CI_Y$, then $Y$ is compatibly split by $\varphi$ if $\varphi(F_*\CI_Y)\subset\CI_Y$; in this case, $\varphi$ induces a Frobenius splitting of $Y$. For a $B^\pm$-variety, we use the notion of a $B^\pm$-canonical splitting in \cite{BK05}*{Definition 4.1.4}. 

By \cite{BK05}*{Proposition 1.2.1}, a Frobenius split scheme is reduced. Scheme-theoretic intersections and unions of compatibly split closed subschemes, as well as the irreducible components of a compatibly split closed subscheme, are again compatibly split. We will repeatedly use these closure properties for the strata considered below.

Suppose that $X$ is smooth. The trace map identifies $H^0(X,\omega_X^{1-p})$ with the space of $\CO_X$-linear maps $F_*\CO_X\to\CO_X$ by \cite{BK05}*{Proposition 1.3.7}. Under this identification, the local criterion in \cite{BK05}*{Theorem 1.3.8} determines whether a section of $\omega_X^{1-p}$ defines a Frobenius splitting. In particular, an anticanonical section $\sigma\in H^0(X,\omega_X^{-1})$ gives a candidate splitting $\sigma^{p-1}\in H^0(X,\omega_X^{1-p})$. If $\sigma^{p-1}$ splits $X$, then the zero scheme of $\sigma$ is compatibly split and reduced by \cite{BK05}*{Proposition 1.3.11}.

Let $D$ be an effective Cartier divisor on $X$. Write $s_D\in H^0(X,\CO_X(D))$ for its canonical section. Following \cite{BK05}*{Definition 1.4.1}, a Frobenius splitting of $X$ relative to $D$, or a $D$-splitting, is an $\CO_X$-linear map $\psi$ for which the composite
\[
F_*\CO_X\xrightarrow{F_*s_D}F_*\CO_X(D)\xrightarrow{\psi}\CO_X
\]
is a Frobenius splitting. If $X$ is smooth, \cite{BK05}*{Theorem 1.4.10} identifies compatibility with an effective divisor $D$ with splitting relative to $(p-1)D$. Consequently, whenever $\sigma^{p-1}$ defines a splitting, it splits $X$ relative to $(p-1)Z(\sigma)$. This is the form used in Sections~\ref{sec:wonderful-splitting} and \ref{sec:equivariant-embeddings}, where the splitting section is obtained from an explicit anticanonical divisor.


We also recall the stable form of splitting relative to a divisor. Following \cite{Smi00}*{Definition 3.5}, $X$ is stably Frobenius split along an effective Cartier divisor $D$ if, for some $e\geq1$, the map $\CO_X\to F_*^e\CO_X(D)$ determined by the canonical section of $D$ admits an $\CO_X$-linear retraction.

Following \cite{Smi00}*{Section~2}, an $F$-finite reduced ring $R$ of characteristic $p>0$ is called strongly $F$-regular if, for every element $c\in R$ not contained in any minimal prime, there is an integer $e\geq1$ such that the $R$-linear map
\[
R\longrightarrow F_*^eR,\qquad 1\longmapsto F_*^e c,
\]
splits. The ring $R$ is strongly $F$-regular if and only if all of its local rings are strongly $F$-regular. Thus we define an scheme $X$ of finite type over $k$ to be strongly $F$-regular if all of its local rings are strongly $F$-regular.

For a projective variety $P$ and an ample invertible sheaf $\CL$, write
\[
R(P,\CL)=\bigoplus_{n\geq0}H^0(P,\CL^{\otimes n}).
\]
By \cite{Smi00}*{Definition 3.2}, $P$ is globally $F$-regular if $R(P,\CL)$ is strongly $F$-regular for some ample $\CL$. 

\begin{theorem}[\cite{Smi00}*{Theorem 3.10}]\label{thm:gloFre}
Let $P$ be an irreducible projective variety over $k$. The following conditions are equivalent:
\begin{enumerate}[label=\textnormal{(\arabic*)}]
\item $P$ is globally $F$-regular.
\item The variety $P$ is stably Frobenius split along an ample effective Cartier divisor $D$, and $P\setminus D$ is strongly $F$-regular.
\item The variety $P$ is stably Frobenius split along every effective Cartier divisor.
\end{enumerate}
\end{theorem}

A globally $F$-regular projective variety is normal, Cohen--Macaulay, and Frobenius split by \cite{Smi00}*{Proposition 4.1}. Moreover, if $\CL$ is a nef invertible sheaf on such a variety, then $H^i(P,\CL)=0$ for every $i>0$ by \cite{Smi00}*{Corollary 4.3}. These consequences will be applied to the closures of the symmetric-space strata in Section~\ref{subsec:global-f-regularity}.

\section{Enveloping varieties}\label{sec:enveloping-variety}

\subsection{The wonderful compactification}\label{subsec:wonderful-geometry}

Retain the notation of Section \ref{sec:preliminaries}, and assume that
$G=G_{\mathrm{sc}}$ is semisimple and simply connected. Let
$\pi_G:G\to G_{\mathrm{ad}}=G/Z(G)$ be the adjoint quotient. The
involution $\theta$ descends to $\theta_{\mathrm{ad}}$ on $G_{\mathrm{ad}}$.
Put
\[
K=G^\theta,\qquad
H=\pi_G^{-1}(G_{\mathrm{ad}}^{\theta_{\mathrm{ad}}}),\qquad
S_{\mathrm{sc}}=G/K,\qquad S_{\mathrm{ad}}=G/H,
\]
where the inverse image defining $H$ is scheme-theoretic. Thus
$S_{\mathrm{ad}}\simeq G_{\mathrm{ad}}/G_{\mathrm{ad}}^{\theta_{\mathrm{ad}}}$.

Let $Q\subset X=X^*(T)$ be the root lattice, and set
\[
\breve X_{\mathrm{sc}}=(1-\theta)X,\qquad
\breve X_{\mathrm{ad}}=(1-\theta)Q,\qquad
\breve X_{\mathrm{sc}}^+=\breve X_{\mathrm{sc}}\cap X^+.
\]
The torus with character lattice $\breve X_{\mathrm{sc}}$ is
$\breve T_{\mathrm{sc}}=T/T^\theta$.
For $\mu\in X$, write $\overline\mu=\mu-\theta(\mu)$. Since $G$ is
simply connected, $X=\bigoplus_{i\in\I}\mathbb Z\omega_i$, where
$\omega_i$ are the fundamental weights. Fix a set
$\Ir\subset\I_\circ$ of representatives of the $\tau$-orbits, and put
\[
Q_\theta^+=\sum_{i\in\Ir}\mathbb N\overline\alpha_i.
\]
The $\overline\alpha_i$, for $i\in\Ir$, form a $\mathbb Z$-basis of
$\breve X_{\mathrm{ad}}$; see \cite{BS25b}*{Sections 2.1--2.3}.

\begin{remark}\label{rem:adjoint-spherical-lattice}
The inclusion $\breve X_{\mathrm{ad}}\subseteq\breve X_{\mathrm{sc}}\cap Q$
need not be an equality. For $G=\mathrm{SL}_2$ with
$\theta(g)=(g^t)^{-1}$, let $\omega$ be the fundamental weight, so
$\alpha=2\omega$ and $\theta(\omega)=-\omega$. Then
\[
\breve X_{\mathrm{sc}}=2\mathbb Z\omega=Q,
\qquad
\breve X_{\mathrm{ad}}=2\mathbb Z\alpha=2Q
\subsetneq\breve X_{\mathrm{sc}}\cap Q.
\]
This discrepancy is a combinatorial shadow of the fact that $H$ may be strictly bigger than $KZ(G)$.
In this example, $Z(G)\subset K$ and $H/K\simeq\mu_2$.
\end{remark}

For $\lambda\in\breve X_{\mathrm{sc}}$, let $\chi_\lambda$ be the
rational $B^-$-eigenfunction on $S_{\mathrm{sc}}$ with left character $-\lambda$,
normalized by $\chi_\lambda(eK)=1$. Then
\[
\chi_{\lambda+\mu}=\chi_\lambda\chi_\mu,
\qquad \chi_\lambda\in k[S_{\mathrm{sc}}]
\quad\text{for }\lambda\in\breve X_{\mathrm{sc}}^+;
\]
see \cite{BS25b}*{Section 2.4 and Lemma 2.2}, with the opposite Borel
convention. 

The quotient $D=H/K$ is a finite diagonalizable group scheme: the map
$h\mapsto h\theta(h)^{-1}$ identifies it with a subgroup scheme of
$Z(G)$. The right $D$-action gives
\begin{equation}\label{eq:enveloping-D}
\begin{gathered}
k[S_{\mathrm{sc}}]^D=k[S_{\mathrm{ad}}],\qquad
X^*(D)=\breve X_{\mathrm{sc}}/\breve X_{\mathrm{ad}},\\
k[S_{\mathrm{sc}}]
=\bigoplus_{c\in X^*(D)}k[S_{\mathrm{sc}}]_c.
\end{gathered}
\end{equation}
For $\lambda\in\breve X_{\mathrm{sc}}$, write $[\lambda]$ for its
class in $X^*(D)$. Under this identification, $\chi_\lambda$ has
right $D$-character $[\lambda]$, and $k[S_{\mathrm{sc}}]_{[\lambda]}$
denotes the corresponding character space. Character spaces are
understood scheme-theoretically, and taking them is exact.

Recall the filtration $k[S_{\mathrm{sc}}]_{\leq\lambda}$ from
Section \ref{subsec:prelim-based-modules}. Write
$\nu\preceq\lambda$ if $N\nu\leq N\lambda$ for some integer $N>0$,
and set
\[
k[S_{\mathrm{sc}}]_{\preceq\lambda}
=\sum_{\substack{\nu\in\breve X_{\mathrm{sc}}^+\\\nu\preceq\lambda}}
k[S_{\mathrm{sc}}]_{\leq\nu},
\]
as in \cite{BS25b}*{(4.1)}. For $\lambda\in\breve X_{\mathrm{sc}}$, define
\begin{equation}\label{eq:enveloping-E}
E_\lambda=\bigl(k[S_{\mathrm{sc}}]_{\preceq\lambda}\bigr)_{[\lambda]}.
\end{equation}

\begin{remark}
In our previous work \cite{BS25b}, we use the full filtered piece $k[S_{\mathrm{sc}}]_{\preceq\lambda}$ in degree
$\lambda$. Our $E_\lambda$ retains only the component with right
$D$-character $[\lambda]$. This character condition implements the
diagonal $D$-quotient needed for the adjoint symmetric space.
\end{remark}

\begin{lemma}\label{lem:enveloping-order}
Let $\lambda,\nu\in\breve X_{\mathrm{sc}}$. Then
\[
[\lambda]=[\nu]
\quad\Longleftrightarrow\quad
\lambda-\nu\in\breve X_{\mathrm{ad}}.
\]
If $\nu\preceq\lambda$, these conditions are equivalent to
$\lambda-\nu\in Q_\theta^+$, and then $\nu\leq\lambda$.
Consequently,
\[
E_\lambda=\bigl(k[S_{\mathrm{sc}}]_{\leq\lambda}\bigr)_{[\lambda]},
\]
and $E_\lambda$ has a good filtration with factors $V(\nu)^*$ for
$\nu\in\breve X_{\mathrm{sc}}^+$ such that
$\lambda-\nu\in Q_\theta^+$, each occurring once.
\end{lemma}

\begin{proof}
The first equivalence follows from \eqref{eq:enveloping-D}.
Suppose $\nu\preceq\lambda$ and $[\lambda]=[\nu]$. Then
$\lambda-\nu\in\breve X_{\mathrm{ad}}\subset Q$, so
$\nu\leq\lambda$. By \cite{BS25b}*{Lemma 2.3},
$\lambda-\nu$ lies in the nonnegative rational cone generated by
the $\overline\alpha_i$. Since these form a $\mathbb Z$-basis of
$\breve X_{\mathrm{ad}}$, its coefficients are nonnegative integers.
Thus $\lambda-\nu\in Q_\theta^+$. The converse follows from
$Q_\theta^+\subset\breve X_{\mathrm{ad}}$.

The good-filtration factor $V(\nu)^*$ of
$k[S_{\mathrm{sc}}]$ is indecomposable, and the $D$-action commutes
with $G$. It therefore has a single $D$-character, namely $[\nu]$,
as its $B^-$-eigenline is represented by $\chi_\nu$.
Exactness of character spaces and the preceding
equivalence therefore retain precisely the factors with
$\lambda-\nu\in Q_\theta^+$. These are also the factors of
$\bigl(k[S_{\mathrm{sc}}]_{\leq\lambda}\bigr)_{[\lambda]}$,
proving the asserted equality.
\end{proof}

Let $j:D\hookrightarrow\breve T_{\mathrm{sc}}$ be dual to the
quotient $\breve X_{\mathrm{sc}}\twoheadrightarrow X^*(D)$. Define
\[
\widehat S=(S_{\mathrm{sc}}\times\breve T_{\mathrm{sc}})/D,
\qquad (x,t)d=(xd,j(d)^{-1}t).
\]
This is a homogeneous space for $G\times\breve T_{\mathrm{sc}}$, with
coordinate ring
\[
k[\widehat S]
=\bigl(k[S_{\mathrm{sc}}]\otimes k[\breve X_{\mathrm{sc}}]\bigr)^D
=\bigoplus_{\lambda\in\breve X_{\mathrm{sc}}}
k[S_{\mathrm{sc}}]_{[\lambda]}e^\lambda.
\]
Here $e^\lambda$ denotes the character $\lambda$ of $\breve T_{\mathrm{sc}}$.
Following \cite{BS25b}*{Section 4.1}, define the affine embedding of $\hat{S}$ by 
\begin{equation}\label{eq:enveloping-algebra}
\widehat E=\operatorname{Spec}
\left(\bigoplus_{\lambda\in\breve X_{\mathrm{sc}}}
E_\lambda e^\lambda\right).
\end{equation}
Multiplication is induced
from $k[\widehat S]$, and the grading defines a right
$\breve T_{\mathrm{sc}}$-action with
$F(xt)=\lambda(t)F(x)$ for homogeneous $F$ of degree $\lambda$.
If $\widetilde E$ denotes the enveloping variety in \cite{BS25b}*{Section 4.1}, then
$k[\widehat E]=k[\widetilde E]^D$ for the diagonal action
$d\cdot(fe^\lambda)=[\lambda](d)^{-1}(R_df)e^\lambda$, where
$(R_df)(x)=f(xd)$.
Thus $\widehat E$ is normal and of finite type, the quotient
$q_D:\widetilde E\to\widehat E$ is finite and surjective, and
$\widehat S\subset\widehat E$ is the open
$G\times\breve T_{\mathrm{sc}}$-orbit.

Choose a regular spherical weight $\eta\in\breve X_{\mathrm{sc}}^+$,
meaning that $\langle\alpha_i^\vee,\eta\rangle>0$ for every
$i\in\I_\circ$. Set
\[
R_\eta=\bigoplus_{n\geq0}E_{n\eta}e^{n\eta},\qquad
\mathbf X_\eta=\operatorname{Proj}R_\eta
=\widehat E/\!\!/_\eta\breve T_{\mathrm{sc}}.
\]

\begin{theorem}\label{thm:enveloping-quotient}
The semistable and stable loci in $\widehat E$ coincide and are
independent of the regular spherical weight $\eta$. Denote this
common locus by $\widehat E^0$. The geometric quotient
\[
\widehat E^0/\breve T_{\mathrm{sc}}=\mathbf X_\eta
\]
is the wonderful compactification $\X$ of $S_{\mathrm{ad}}$.
\end{theorem}

\begin{proof}
The argument of \cite{BS25b}*{Theorem 4.10} applies with the weight
monoid
\[
\{(\nu,\lambda)\in\breve X_{\mathrm{sc}}^+
\times\breve X_{\mathrm{sc}}\mid\lambda-\nu\in Q_\theta^+\}. \qedhere
\]

\end{proof}

\begin{remark}\label{rem:affine-embedding-constructions}
Although $\widehat S$ need not be a symmetric space for
$G\times\breve T_{\mathrm{sc}}$ in the fixed-point-subgroup sense,
the constructions of affine embeddings, their orbit closures, and
dual canonical bases in \cite{BS25b}*{Sections 3.1--3.3 and 4.2}
apply with the character condition $[\nu]=[\lambda]$.
The essential point is that, within each character class,
Lemma \ref{lem:enveloping-order} identifies the order
$\lambda-\nu\in Q_\theta^+$ with ordinary dominance, which defines
the dual-canonical-basis filtration. One uses the corresponding
character components and the same constructions by closed
saturated submonoids and ideals. We omit the details, since they
repeat the arguments of \cite{BS25b} with this restriction.
\end{remark}

Restriction to the closed orbit embeds
$\operatorname{Pic}(\X)$ in $X$ \cite{DCS99}*{Theorem 4.2(ii)}.
We use the fiber-character convention: each $t\in T$ acts on the one-dimensional fiber
$(\CL_\mu)_{eP}$ by the scalar $\mu(t)$.

\begin{prop}\label{prop:wc}
\begin{enumerate}
\item The lattice $\breve X_{\mathrm{sc}}$ is contained in $\operatorname{Pic}(\X)$.
\item The canonical sheaf is
\[
\omega_\X\cong\CL_{\overline\rho+\sum_{i\in\Ir}\overline\alpha_i},
\qquad \rho=\sum_{i\in\I}\omega_i.
\]
\item The boundary is
\[
\partial\X=\X\setminus S_{\mathrm{ad}}=\bigcup_{i\in\Ir}D_i,
\qquad \CO_\X(D_i)\cong\CL_{-\overline\alpha_i},
\]
where each $D_i$ is a $G$-stable prime divisor.
\item The intersection $\bigcap_{i\in\Ir}D_i$ is the unique closed
$G$-orbit, isomorphic to $G/P$, where $P\supset B^+$ is the standard
parabolic subgroup associated with $\I_\bullet$.
\end{enumerate}
\end{prop}

These statements follow from \cite{DCS99}*{Theorems 3.9 and 4.8,
and the proof of Theorem 5.9}, with the line-bundle indices reversed
to match the fiber-character convention.

\subsection{Line bundles and semi-invariant functions}\label{subsec:enveloping-preliminary-results}

\begin{lemma}\label{lem:Ttorsor}
The quotient morphism
\[
\pi:\widehat E^0\longrightarrow\X
\]
is a Zariski-locally trivial $\breve T_{\mathrm{sc}}$-torsor.
More precisely, every $x\in\widehat E^0$ has a $\breve T_{\mathrm{sc}}$-stable affine
neighbourhood of the form
\[
\pi^{-1}(W)\cong W\times \breve T_{\mathrm{sc}}
\]
over an affine open subset $W\subset\X$.
\end{lemma}

\begin{proof}
Recall that $q_\eta=\chi_\eta e^\eta\in k[\widehat E]$ and that
$D(gq_\eta)$, for $g\in G(k)$, cover $\widehat E^0$.
For $\lambda\in\breve X_{\mathrm{sc}}^+$, put
$q_\lambda=\chi_\lambda e^\lambda$.
We first prove that the scheme-theoretic stabilizers are trivial.
For any $\nu\in\breve X_{\mathrm{sc}}$, choose $n$ sufficiently
large that $n\eta\pm\nu\in\breve X_{\mathrm{sc}}^+$. On
$D(q_\eta)$, the functions
\[
q_\nu=\frac{q_{n\eta+\nu}}{q_\eta^n},\qquad
q_{-\nu}=\frac{q_{n\eta-\nu}}{q_\eta^n}
\]
are regular and homogeneous of degrees $\nu$ and $-\nu$.
Multiplicativity of the eigenfunctions gives
\[
q_{n\eta+\nu}q_{n\eta-\nu}=q_{2n\eta}=q_\eta^{2n},
\]
so $q_\nu q_{-\nu}=1$. Thus both functions are units on
$D(q_\eta)$.

Let $x$ be a geometric point of $D(q_\eta)$, and let
$N=\operatorname{Stab}_{\breve T_{\mathrm{sc}}}(x)$ be its
scheme-theoretic stabilizer. The homogeneity identity
$q_\nu(xt)=\nu(t)q_\nu(x)$ restricts to
\[
q_\nu(x)\,\nu|_N=q_\nu(x)
\]
as an equality of regular functions on $N$. Since $q_\nu(x)$ is
an invertible scalar, $\nu|_N=1$. This holds for every
$\nu\in\breve X_{\mathrm{sc}}$. In particular, for a basis
$\varepsilon_1,\ldots,\varepsilon_r$ of this lattice, we have
\[
N\subseteq\bigcap_{j=1}^r\ker(\varepsilon_j)=\{1\}
\]
as group schemes, and hence $N$ is trivial. The $G$-action
commutes with the torus action, so the same holds on every
$D(gq_\eta)$. These opens cover $\widehat E^0$.

Let $x\in\widehat E^0$. Triviality of its stabilizer means that the
degrees of homogeneous functions nonzero at $x$ generate $\breve X_{\mathrm{sc}}$.
Choose finitely many such functions $f_1,\ldots,f_m\in k[\widehat E]$.
Since $x$ is semistable, also choose $h\in k[\widehat E]_{n\eta}$ with
$h(x)\ne0$, for some $n>0$, and set $F=hf_1\cdots f_m$.
Then $D(F)$ is $\breve T_{\mathrm{sc}}$-stable, contains $x$, and is contained in
$\widehat E^0$. Each $f_i$ is a unit in $k[\widehat E]_F$.
For a basis $\varepsilon_1,\ldots,\varepsilon_r$ of $\breve X_{\mathrm{sc}}$, choose
Laurent monomials $u_j$ in the $f_i$ of degree $\varepsilon_j$.
Writing $B=(k[\widehat E]_F)_0$, multiplication gives an isomorphism
of graded algebras
\[
B\otimes_k k[\breve X_{\mathrm{sc}}]\xrightarrow{\sim}k[\widehat E]_F,
\qquad b\otimes e^{\sum_jn_j\varepsilon_j}
\longmapsto b\prod_j u_j^{n_j}.
\]
Indeed, dividing a homogeneous element by the unit of its degree
leaves an element of $B$, and uniqueness follows from the grading.

Since $\pi$ is a geometric quotient, the stable open $D(F)$ is
saturated and $W=\pi(D(F))$ is open. Its affine quotient is
$W=\operatorname{Spec}B$. The displayed isomorphism therefore gives
\[
\pi^{-1}(W)=D(F)\cong W\times \breve T_{\mathrm{sc}},
\]
with $\pi$ equal to the first projection.
\end{proof}

\begin{lemma}\label{lem:codim2}
The complement $\widehat E\setminus\widehat E^0$ has codimension
at least two. Consequently, restriction induces an isomorphism
\[
k[\widehat E]\xrightarrow{\sim}
\Gamma(\widehat E^0,\CO_{\widehat E^0}).
\]
\end{lemma}

\begin{proof}
We first apply the orbit argument to the variety
$\widetilde E$. An irreducible component of
$\widetilde E\setminus\widetilde E^0$ is stable under the connected
group $G\times \breve T_{\mathrm{sc}}$. By \cite{BS25b}*{Proposition 4.9}, its prime
ideal is $\mathfrak p_{J_1,J_2}$ for an essential pair
$(J_1,J_2)$. The instability criterion in
\cite{BS25b}*{Theorem 4.10} gives $J_1\ne\Ir$, and essentiality
forces $J_2\ne\Ir$. The defining faces in
\cite{BS25b}*{Definition 4.7} give strict inclusions
\[
(0)\subsetneq\mathfrak p_{\Ir,J_2}
\subsetneq\mathfrak p_{J_1,J_2}.
\]
The first inclusion is strict because a boundary direction
$(0,\overline\alpha_j)$ with $j\notin J_2$ lies in the first nonzero ideal.
The second is strict because a positive multiple of
$(\overline\omega_i,\overline\omega_i)$ with $i\notin J_1$ lies
in the latter ideal but not the former. Thus the unstable
complement in $\widetilde E$ has codimension at least two.

The finite surjective morphism $q_D:\widetilde E\to\widehat E$
satisfies $q_D^{-1}(\widehat E^0)=\widetilde E^0$. It preserves
dimensions of closed subsets and maps the unstable complement
onto $\widehat E\setminus\widehat E^0$, proving the same
codimension bound for $\widehat E$. The restriction isomorphism
then follows from normality of $\widehat E$.
\end{proof}

\begin{thm}\label{thm:linebundles}
\begin{enumerate}
\item For every $\mu\in \breve X_{\mathrm{sc}}$, the associated bundle
\[
\widehat E^0\times^{\breve T_{\mathrm{sc}}}k_\mu\longrightarrow\X,
\qquad (x,a)t=(xt,\mu(t)^{-1}a),
\]
is isomorphic to $\CL_\mu$.
\item There are multiplication-compatible identifications
\[
H^0(\X,\CL_\mu)
\cong k[\widehat E]^{(-\mu)}
=E_{-\mu}e^{-\mu}
\cong\bigl(k[S_{\mathrm{sc}}]_{\leq-\mu}\bigr)_{[-\mu]}.
\]
Here the superscript $(-\mu)$ denotes the semi-invariant functions
satisfying $F(xt)=\mu(t)^{-1}F(x)$.
\item If $0\ne\phi\in E_\mu$ corresponds to
$s_\phi\in H^0(\X,\CL_{-\mu})$, then
\[
\pi^{-1}\bigl(Z(s_\phi)\bigr)
=Z(\phi e^\mu)\cap\widehat E^0
\]
as closed subschemes. Consequently, the zero scheme, and hence the
divisor, of $s_\phi$ is the scheme-theoretic image of this inverse
image under $\pi$.
\end{enumerate}
\end{thm}

\begin{proof}
Lemma \ref{lem:Ttorsor} makes the associated bundle invertible.
On the chart $D(q_\eta)$, the homogeneous unit $q_\lambda$ gives
a nowhere vanishing section of the bundle with index $-\lambda$.
Its left $B^-$-weight is $-\lambda$, so its value at the $T$-fixed
point $eP$ of the closed orbit has fiber character $-\lambda$.
Injectivity of restriction to $G/P$
\cite{DCS99}*{Theorem 4.2(ii)} identifies this bundle with
$\CL_{-\lambda}$, proving (1).

A section of the associated bundle with index $\mu$ is a regular
function $F$ on $\widehat E^0$ with $F(xt)=\mu(t)^{-1}F(x)$.
Lemma \ref{lem:codim2} identifies these functions with the
degree $-\mu$ component of $k[\widehat E]$.
Lemma \ref{lem:enveloping-order} then gives (2); multiplication is
the multiplication of semi-invariant functions.

For (3), on a trivializing chart $W$, choose a homogeneous unit
$u_\mu$ of degree $\mu$ on $\pi^{-1}(W)$. Write
$\phi e^\mu=(\pi^*a)u_\mu$ for $a\in\CO_\X(W)$.
The section $s_\phi|_W$ is represented by $a$, and $u_\mu$ is
invertible, so the asserted equality holds scheme-theoretically.
The charts cover $\X$, and faithful flatness of $\pi$ identifies
the scheme-theoretic image with $Z(s_\phi)$.
\end{proof}

\begin{remark}
If there are exceptional roots in the sense of \cite{DCS99}*{Section~4}, the lattice $\breve X_{\mathrm{sc}}$ need not be all of
$\operatorname{Pic}(\X)$. Such additional classes require
semi-invariant functions on $G$, rather than functions on $G/K$. 
\end{remark}

\section{Splitting of the wonderful compactification}\label{sec:wonderful-splitting}

\subsection{Open charts and distinguished sections}\label{subsec:wonderful-distinguished-sections}

Retain the simply connected group and the character-restricted construction
of Section~\ref{sec:enveloping-variety}, and assume
$\operatorname{char}k\ne2$.

For $v\le w$ in $W$, write
\[
\mathring{\mathcal{R}}_{v,w}
=(B^-\dot vB^+/B^+)\cap(B^+\dot wB^+/B^+),
\qquad
\mathcal{R}_{v,w}=\overline{\mathring{\mathcal{R}}_{v,w}}
\]
to denote the open and closed Richardson varieties. Write
$\pi_P:G/B^+\to G/P$ for the natural projection. Its images of
Richardson varieties are called projected Richardson varieties;
see \cite{KLS14}*{Section~3}.

Define the order-preserving involution
\[
\kappa(v)=w_0\tau(v^{-1})w_0\qquad(v\in W).
\]
It fixes $w^\bullet$. For a simple reflection we write
$\kappa(s_i)=s_{\kappa(i)}$, so $\kappa(i)=(\tau i)^*$.

\begin{proposition}\label{prop:chart-richardson}
Let $v\le w^\bullet$ and $\breve T_{\mathrm{sc}}=T/T^\theta$.
The chart parametrizations
\[
\begin{aligned}
U_P^-\times\breve T_{\mathrm{sc}}&\xrightarrow{\sim}B^-K/K,
&(x,tT^\theta)&\longmapsto xtK,\\
U_P^+\times\breve T_{\mathrm{sc}}&\xrightarrow{\sim}B^+K/K,
&(x,tT^\theta)&\longmapsto x\dot w^\bullet tK
\end{aligned}
\]
induce the isomorphisms
\[
\begin{aligned}
\mathring S_v\cap(B^-K/K)
&\xrightarrow{\sim}\mathring{\mathcal{R}}_{w_\bullet,vw_0}\times\breve T_{\mathrm{sc}},\\
xtK&\longmapsto (x\dot w_\bullet B^+,tT^\theta),\\[4pt]
\mathring S_v\cap(B^+K/K)
&\xrightarrow{\sim}\mathring{\mathcal{R}}_{\kappa(v),w^\bullet}\times\breve T_{\mathrm{sc}},\\
x\dot w^\bullet tK&\longmapsto (x\dot w^\bullet B^+,tT^\theta).
\end{aligned}
\]

(The corresponding
statements hold on the adjoint symmetric space as well. )
\end{proposition}

\begin{proof}
The $U_P^-\times\breve T_{\mathrm{sc}}$ chart
parametrization is the local structure theorem
\cite{DCS99}*{Proposition~3.8}. Its $\dot w^\bullet$-translate
has unipotent coordinate $U_P^+$ and is the open $B^+$-orbit
$B^+K/K$, giving the other parametrization.

For $x\in U_P^-$, one has $\theta(x)^{-1}\in U_P^+
\subset\theta(B^-)$. Therefore
\begin{align*}
xtK\in\mathring S_v
&\Longleftrightarrow
xt\theta(t)^{-1}\theta(x)^{-1}
\in B^+\dot v\dot w^\bullet\theta(B^-)\\
&\Longleftrightarrow x\in B^+\dot v\dot w^\bullet\theta(B^-)\\
&\Longleftrightarrow x\dot w_\bullet\in B^+\dot v\dot w_0B^+.
\end{align*}
Here the torus factors are absorbed into $\theta(B^-)$.
The map $x\mapsto x\dot w_\bullet B^+$ identifies $U_P^-$
with $B^-\dot w_\bullet B^+/B^+$, proving the first restriction.

For $x\in U_P^+$, the identity $\theta(w^\bullet)=w^\bullet$
gives
\[
\iota(x\dot w^\bullet tK)=x\,a(t)\,\theta(x)^{-1},
\qquad a(t)=\dot w^\bullet t\theta(t)^{-1}
\theta(\dot w^\bullet)^{-1}\in T.
\]
Absorbing the left $B^+$-factors, and then applying inversion
and $\theta$, yields
\begin{align*}
x\dot w^\bullet tK\in\mathring S_v
&\Longleftrightarrow
\theta(x)^{-1}(\dot w^\bullet)^{-1}\in B^+\dot vB^-\\
&\Longleftrightarrow
\dot w^\bullet\theta(x)\in B^-\dot v^{-1}B^+\\
&\Longleftrightarrow
x\dot w^\bullet\in B^-\dot{\kappa(v)}B^+.
\end{align*}
For the last equivalence, conjugate the $\theta$-image by
$\dot w^\bullet$ and use
\[
(w^\bullet)^{-1}\theta(v^{-1})w^\bullet
=w_0\tau(v^{-1})w_0=\kappa(v).
\]
Factors in $T$ arising from the chosen representatives do not
affect these double cosets. We conclude that  $x\dot w^\bullet B^+ \in \mathring{\mathcal{R}}_{\kappa(v),w^\bullet}$. 
\end{proof}

\begin{lemma}\label{le:flo}
For $\mu\in X^+$, one has
$\iota^*(k[G]_{\le\mu})\subseteq k[S_{\mathrm{sc}}]_{\le\overline\mu}$.
\end{lemma}

\begin{proof}
For $f\in k[G]$, write $\wt(f)=(\lambda,\nu)\in X\times X$ if
$f(t_1^{-1}gt_2)=\lambda(t_1)\nu(t_2)f(g)$ for $t_1,t_2\in T$.
For $f\in k[S_{\mathrm{sc}}]$, write $\wt_\theta(f)=\lambda$ if
$f(t^{-1}gK)=\lambda(t)f(gK)$ for $t\in T$.
Then
\[
\wt_\theta(\iota^*f)=\lambda+\theta\nu
\qquad\text{when }\wt(f)=(\lambda,\nu)
\]
whenever the pullback is nonzero.

Recall that $k[G]_{\le\mu}$ has a good $G\times G$-filtration
with sections $V(\gamma)^*\boxtimes V(-w_0\gamma)^*$ for
$\gamma\in X^+$, $\gamma\le\mu$. If $f\in k[G]_{\le\mu}$
has weight $(\lambda,\nu)$, then
\[
-\mu\le x\lambda,\qquad x\nu\le\mu\qquad(x\in W).
\]
Consequently,
\[
\wt_\theta(\iota^*f)=\lambda-w_\bullet\tau\nu
\ge-\mu-\tau\mu.
\]
Thus every weight of $M:=\iota^*(k[G]_{\le\mu})$ is at least
$-\mu-\tau\mu$.

We first prove the following assertion.

\emph{(a) Let $\lambda\in\breve X_{\mathrm{sc}}$, and write
$\lambda=\sum_{i\in\I}n_i\alpha_i$, with $n_i\in\mathbb Q$.
If $n_i\ge0$ for $i\in\I_\circ$, then $n_i\ge0$ for every
$i\in\I$.}

Write $\lambda=\lambda_\circ+\lambda_\bullet$, where
\[
\lambda_\circ\in\mathbb Q_{\ge0}[\alpha_i\mid i\in\I_\circ],
\qquad
\lambda_\bullet\in\mathbb Q[\alpha_i\mid i\in\I_\bullet].
\]
Set
\[
\lambda_\circ'=w_\bullet\tau\lambda_\circ-\tau\lambda_\circ
\in\mathbb Q_{\ge0}[\alpha_i\mid i\in\I_\bullet].
\]
Since $\theta$ fixes the black roots,
$\theta\lambda=-\tau\lambda_\circ-\lambda_\circ'+\lambda_\bullet$.
Comparing black-root coefficients in $\theta\lambda=-\lambda$
gives $2\lambda_\bullet=\lambda_\circ'$. This proves~(a).

We next prove:

\emph{(b) If $\lambda\in\breve X_{\mathrm{sc}}$ and
$\lambda\le\mu+\tau\mu$, then $\lambda\le\overline\mu$.}

Put $\delta=\overline\mu-\lambda\in\breve X_{\mathrm{sc}}$.
Since
\[
\overline\mu-(\mu+\tau\mu)
=w_\bullet\tau\mu-\tau\mu
\in\mathbb Z[\alpha_i\mid i\in\I_\bullet],
\]
one has $\delta\in Q$, and the white-root coefficients of
$\delta$ equal those of $\mu+\tau\mu-\lambda$, hence are
nonnegative. By~(a), all simple-root coefficients of $\delta$
are nonnegative. They are integers because $\delta\in Q$.
Thus $\lambda\le\overline\mu$, proving~(b).

Return to $M$. Choose a total order
$\gamma_1,\gamma_2,\ldots$ on $\breve X_{\mathrm{sc}}^+$
refining dominance, with the finite downward-closed set
\[
\{\gamma\in\breve X_{\mathrm{sc}}^+\mid\gamma\le\overline\mu\}
\]
first. Put
$F_N=\sum_{k=1}^N k[S_{\mathrm{sc}}]_{\le\gamma_k}$.
By \cite{BS24}*{Theorem~4.9}, with the right-quotient convention,
$F_N/F_{N-1}\cong V(\gamma_N)^*$.
If $M=0$ there is nothing to prove. Otherwise, let $N$ be minimal
such that $M\subseteq F_N$, and consider its image $M_N$ under
\[
F_N\longrightarrow F_N/F_{N-1}\cong V(\gamma_N)^*.
\]
The subspace $M$ is $G$-stable for the twisted diagonal action,
so $M_N$ is a nonzero $G$-submodule. It contains a vector of weight
$-\gamma_N$: its dual is a nonzero quotient of $V(\gamma_N)$,
in which the highest weight generator must survive.
The preliminary weight bound gives $\gamma_N\le\mu+\tau\mu$.
Applying~(b) to the spherical label $\gamma_N$ yields
$\gamma_N\le\overline\mu$. By the chosen order, every
$\gamma_k$ with $k\le N$ also satisfies this inequality.
Therefore
\[
M\subseteq F_N\subseteq k[S_{\mathrm{sc}}]_{\le\overline\mu},
\]
as required.
\end{proof}

For $i\in\I$, define
\begin{equation}\label{eq:phi}
\varphi_i(gK)=v_{\omega_i}^*\!
\left(g\theta(g)^{-1}v_{w^\bullet\omega_i}\right)
=\iota^*(\Delta_{\omega_i,w^\bullet\omega_i})(gK).
\end{equation}
Lemma~\ref{le:flo} gives the filtration bound for $\varphi_i$.
For $h\in H$, the central element $z_h=h\theta(h)^{-1}$ satisfies
\[
\varphi_i(ghK)=\omega_i(z_h)\varphi_i(gK).
\]
By \eqref{eq:enveloping-D}, this is the character $[\overline{\omega_i}]$ of
$D$, as an identity of group-scheme actions. Consequently,
\[
\varphi_i\in E_{\overline{\omega_i}}
=\bigl(k[S_{\mathrm{sc}}]_{\le\overline{\omega_i}}\bigr)_{[\overline{\omega_i}]}
\cong H^0(\X,\CL_{-\overline{\omega_i}}),
\]
where the last identification is Theorem~\ref{thm:linebundles}.
We use the same symbol for the resulting section.
The support of its zero scheme on $S_{\mathrm{sc}}$ is
$S_{s_{\kappa(i)}}$, where $\kappa(i)=(\tau i)^*$.
This support descends to $S_{\mathrm{ad}}$;
we use the same stratum notation there and write $\X_u$ for its
closure in $\X$. Reducedness will follow in
Corollary~\ref{cor:si}.

\subsection{Restriction to the closed orbit}\label{subsec:wonderful-closed-orbit}

Write $Y=G/P$. We use
$\operatorname{Pic}(Y)=\{\mu\in X\mid
\langle\alpha_j^\vee,\mu\rangle=0\ (j\in \I_{\bullet})\}$ with the
fiber-character convention of Section~\ref{sec:enveloping-variety}.
For $i\in\I$, let $\overline{\varphi_i}$ be the restriction of
$\varphi_i$ to $Y$.

\begin{prop}\label{prop:dv}
(1) For $i\in\I$, the element $\overline{\varphi_i}\in H^0(G/P,\CL_{-\overline{\omega_i}})$ is nonzero.

    (2) For $i\in\I_\circ$, one has $Z(\overline{\varphi_i})=\overline{B^-s_iP/P}\cup \overline{B^+w_0s_{\tau i}P/P} $.

    (3) For $i\in \I_\bullet$, one has $Z(\overline{\varphi_i})=\pi_P(\overline{B^-s_iB^+/B^+}\cap \overline{B^+w^\bullet B^+/B^+})$. Here $\pi_P:G/B^+\rightarrow G/P$ is the canonical projection.
\end{prop}

\begin{proof}
The local structure theorem gives the charts
\[
\X^-\cong U_P^-\times Z,\qquad
\X^+=\dot w^\bullet\X^-\cong U_P^+\times\dot w^\bullet Z,
\qquad Z\cong\mathbb A^r;
\]
see \cite{DCS99}*{Propositions~3.6 and~3.8}. The toric variety
$Z$ has open torus $\breve T_{\mathrm{ad}}$ and a fixed point $z_0$
with stabilizer $P$. Identify $Y=G/P$ by $gP\mapsto gz_0$.
Then
\[
Y\cap\X^-=B^-P/P,\qquad
Y\cap\X^+=B^+\dot w^\bullet P/P.
\]

We first prove~(1). By Theorem~\ref{thm:linebundles} and its proof,
the section $\chi_{\overline{\omega_i}}$ is defined and nowhere vanishing
on $\X^-$. Thus $\dot w^\bullet\chi_{\overline{\omega_i}}$ is defined
on $\X^+$ and does not vanish there.
For $x\in U_P^+$ and $t\in T$, the definition of $\varphi_i$ gives
\[
\frac{\varphi_i(x\dot w^\bullet tK)}
{(\dot w^\bullet\chi_{\overline{\omega_i}})(x\dot w^\bullet tK)}
=c_i\Delta_{\omega_i,w^\bullet\omega_i}(x),
\qquad c_i\in k^\times.
\]
Indeed, $\theta(x)^{-1}$ fixes $v_{w^\bullet\omega_i}$ and the
common torus factor is $\overline{\omega_i}(t)$. The quotient descends under
$D$ because its numerator and denominator have the same character.
It is independent of the torus coordinate, so the same formula
holds on $Y\cap\X^+$. It is nonzero whenever
$x\dot w^\bullet B^+\in\mathring{\mathcal{R}}_{e,w^\bullet}$,
a nonempty open Richardson variety. Hence
$\overline{\varphi_i}\ne0$, proving~(1).

We now prove~(2) and~(3).
For $v\le w^\bullet$, let $\mathring C_v^\pm\subset U_P^\pm$
be the Richardson loci in the unipotent coordinates of
Proposition~\ref{prop:chart-richardson}. On the adjoint open orbit,
\[
\mathring S_v\cap\X^\pm
\cong\mathring C_v^\pm\times\breve T_{\mathrm{ad}}.
\]
Keeping the unipotent coordinate fixed and taking the torus factor
to $z_0$, or to $\dot w^\bullet z_0$, gives the points $xP$, or
$x\dot w^\bullet P$, in $\X_v\cap Y$. Taking closures in $Y$ yields
\begin{equation}\label{eq:pip}
\pi_P(\mathcal{R}_{w_\bullet,vw_0})\subseteq\X_v\cap Y,
\qquad
\pi_P(\mathcal{R}_{\kappa(v),w^\bullet})\subseteq\X_v\cap Y.
\end{equation}
Since $\X_{s_{\kappa(i)}}\subseteq Z(\varphi_i)$, these inclusions
immediately imply
\begin{equation}\label{eq:phi-divisor-bound}
\begin{gathered}
\operatorname{div}(\overline{\varphi_i})\ge A_i,\\
A_i:=\begin{cases}
\pi_P(\mathcal R_{s_i,w^\bullet})
+\pi_P(\mathcal R_{w_\bullet,w_0s_{\tau i}}),&i\notin \I_{\bullet
},\\
\pi_P(\mathcal R_{s_i,w^\bullet}),&i\in \I_{\bullet
}.
\end{cases}
\end{gathered}
\end{equation}
Here each projected Richardson variety has its reduced divisor
structure, and an empty variety contributes zero. For $i\in \I_{\bullet}$,
the two projected Richardson divisors coincide by
\cite{KLS14}*{Lemma~3.1}:
\[
\pi_P(\mathcal R_{s_i,w^\bullet})
=\pi_P(\mathcal R_{w_\bullet,w_0s_{\tau i}}).
\]

The $A_i$ together contain each projected Richardson divisor of
$G/P$ exactly once. Thus, by \cite{KLS14}*{Lemma~5.4},
\[
\Delta:=\sum_{i\in\I}A_i
=\sum_{i\in\I}\pi_P(\mathcal R_{s_i,w^\bullet})
+\sum_{i\notin J}\pi_P(\mathcal R_{w_\bullet,w_0s_i})
\]
is anticanonical. On the other hand,
$\overline\rho=\sum_i\overline{\omega_i}
=\rho+w_\bullet\rho=2\rho-2\rho_{\I_{\bullet}}$, so
\[
0\ne\prod_{i\in\I}\overline{\varphi_i}
\in H^0(Y,\CL_{-\overline\rho}|_Y)
=H^0(Y,\omega_Y^{-1});
\]
see \cite{BK05}*{proof of Theorem~4.1.15, equation~(1)}.
Consequently,
\[
\operatorname{div}\!\left(\prod_i\overline{\varphi_i}\right)
\ge\Delta.
\]
Both divisors are anticanonical. Their difference is effective and
linearly equivalent to zero on the projective integral variety $Y$,
and hence vanishes. Therefore
$\operatorname{div}(\overline{\varphi_i})=A_i$ for every $i$.
Since the $A_i$ are reduced divisors on the smooth variety $Y$,
these equalities identify the zero schemes and prove~(2) and~(3).
\end{proof}

\subsection{The canonical Frobenius splitting}\label{subsec:wonderful-canonical-splitting}

Assume $\operatorname{char}k=p>2$.

\begin{theorem}\label{thm:WCsplit}
The wonderful compactification $\X$ admits a $B^+$-canonical Frobenius splitting that compatibly splits each divisor $\X_{s_i}$, for $i\in\I$, and each boundary divisor $D_i$, for $i\in\Ir$.
\end{theorem}

\begin{proof}
For each $i\in\Ir$, Proposition \ref{prop:wc}(3) provides a
$G$-invariant section
\[
f_i\in H^0(\X,\CL_{-\overline{\alpha_i}})
\]
whose zero divisor is $D_i$. Proposition \ref{prop:wc}(2) gives
\[
\sigma=\left(\prod_{i\in\Ir}f_i\right)
\left(\prod_{i\in\I}\varphi_i\right)
\in H^0(\X,\omega_\X^{-1}).
\]
The character restriction is retained throughout: the second product
belongs to $E_{\overline\rho}$ under
Theorem~\ref{thm:linebundles}.

By Proposition~\ref{prop:dv}, the restriction
$\prod_i\overline{\varphi_i}$ has divisor $\Delta$, the reduced
union of all projected Richardson divisors of $Y=G/P$.
The splitting in \cite{KLS14}*{Corollary~4.3} compatibly splits
these divisors, so its section of $\omega_Y^{1-p}$ vanishes along
$(p-1)\Delta$ by \cite{BK05}*{Exercise~1.3.E.3}.
Since $\Delta$ is anticanonical by \cite{KLS14}*{Lemma~5.4},
that section is a nonzero scalar multiple of
$(\prod_i\overline{\varphi_i})^{p-1}$.
The boundary divisors of $\X$ are smooth with normal crossings
and have intersection $Y$ by \cite{DCS99}*{Theorem~3.9}.
Successive adjunction along them and
\cite{BK05}*{Exercise~1.3.E.4} show that $\sigma^{p-1}$ defines
a Frobenius splitting after a nonzero scalar normalization.

By \cite{BK05}*{Proposition~1.3.11}, the zero scheme of $\sigma$
is reduced and compatibly split. Each $Z(\varphi_i)$ is therefore
reduced. Its support on the open orbit is
$S_{s_{\kappa(i)}}$. It has no boundary-divisor component:
every boundary divisor contains $Y$, whereas
$\overline{\varphi_i}\ne0$ by Proposition~\ref{prop:dv}(1).
Thus $Z(\varphi_i)=\X_{s_{\kappa(i)}}$. Since $\kappa$
permutes $\I$, the splitting is compatible with every $\X_{s_i}$
and every $D_i$.

We verify $B^+$-canonicality directly. The left $T$-weight of
$\varphi_i$ is
\[
-\omega_i+\theta(w^\bullet\omega_i)=-\omega_i-w_0\omega_{\tau i}.
\]
Since $w_0\rho=-\rho$, their sum is zero, and $\sigma$ is
$T$-invariant. Fix a simple root subgroup $x_j(a)\subset B^+$.
The vector $v_{w^\bullet\omega_i}$ is fixed by $\theta(U^+)$, because
$\theta(B^+)=\dot w^\bullet B^+(\dot w^\bullet)^{-1}$. Hence on $S_{\mathrm{sc}}$,
\[
(x_j(a)\cdot\varphi_i)(gK)
=v_{\omega_i}^*\!\left(x_j(-a)g\theta(g)^{-1}v_{w^\bullet\omega_i}\right).
\]
The simple-root string through the lowest weight vector
$v_{\omega_i}^*\in V(\omega_i)^*$ has length
$\langle\alpha_j^\vee,\omega_i\rangle=\delta_{ij}$.
Thus this expression has degree at most $\delta_{ij}$ in $a$.
The section identifications of Theorem~\ref{thm:linebundles} are
$G$-equivariant, and the $f_i$ are $G$-invariant. Consequently,
$x_j(a)\cdot\sigma$ has degree at most one in $a$, and
$x_j(a)\cdot\sigma^{p-1}$ has degree at most $p-1$.
Under the equivariant identification
$H^0(\X,\omega_\X^{1-p})\cong\operatorname{End}_F(\X)$ of
\cite{BK05}*{Lemma~4.1.14}, these are exactly the torus-invariance
and divided-power conditions of \cite{BK05}*{Lemma~4.1.6}.
The normalized splitting is therefore $B^+$-canonical.
\end{proof}

\begin{remark}
The splitting of $G/P$ used in \cite{KLS14}*{Corollary~4.3} is the
pushforward of the standard splitting of $G/B^+$, and is
$B^+$-canonical; see \cite{BK05}*{Theorem~4.1.15}.
The direct verification above also proves canonicality on $\X$
without lifting a map through restriction to the closed orbit.
\end{remark}

\begin{cor}\label{cor:si}
(1) For $i\in\I$, the zero scheme $Z(\varphi_i)$ is reduced and
\[
Z(\varphi_i)=\X_{s_{\kappa(i)}},\qquad
\CO_\X(\X_{s_{\kappa(i)}})\cong\CL_{-\overline{\omega_i}}.
\]

(2) $\cup_{i\in\I}(\X_{s_i}\cap G/P)$ is the union of all projected Richardson divisors of $G/P$.
\end{cor}

\section{Equivariant embeddings}\label{sec:equivariant-embeddings}

\subsection{Frobenius splittings of equivariant embeddings}\label{subsec:embeddings-splitting}

Return to the general connected reductive group $G$ and symmetric space
$S=G/K$ of Section~\ref{sec:preliminaries}.  
For an equivariant embedding $X$ of $S$, let $X_u$ be the 
closure of $S_u$ in $X$.
Write $\partial X$ for the reduced boundary divisor, the sum of
the prime divisors in $X\setminus S$.

Put $G_{\mathrm{ad}}=G/Z(G)$ and
$S_{\mathrm{ad}}=G_{\mathrm{ad}}/G_{\mathrm{ad}}^{\theta_{\mathrm{ad}}}$.
Apply
Sections~\ref{sec:enveloping-variety} and \ref{sec:wonderful-splitting}
to the simply connected cover of $G_{\mathrm{ad}}$, with the lifted
involution. This gives the wonderful compactification $\X$ of
$S_{\mathrm{ad}}$ and the sections $\varphi_i$ in line bundles on $\X$.
Denote the natural map $S\to S_{\mathrm{ad}}$ by $\pi$; the same
symbol will denote its extension to a toroidal embedding.

\begin{lemma}\label{lem:embeddings-toroidal-map}
Let $X$ be a nonsingular complete toroidal embedding of $S$. Then $\pi$ extends to a $G$-equivariant morphism $\pi:X\to\X$. If $Y\subset X$ and $\Y\subset\X$ are closed $G$-orbits, then $\pi|_Y:Y\to\Y$ is an isomorphism. Moreover, if
\[
\widetilde\omega_X=\omega_X\otimes\CO_X(\partial X),
\qquad
\widetilde\omega_{\X}=\omega_{\X}\otimes\CO_{\X}(\partial\X),
\]
then
\[
\widetilde\omega_X\simeq\pi^*\widetilde\omega_{\X}.
\]
\end{lemma}

\begin{proof}
The extension of $\pi$ follows from \cite{Kno91}*{Theorem 4.1}, the assertion on closed orbits from \cite{Tan18}*{Lemma 1.2}, and the identity of logarithmic canonical sheaves from \cite{Tan18}*{Lemma 1.4}.
\end{proof}

\begin{theorem}\label{thm:main}
Assume $\operatorname{char}k=p>2$, and let $X$ be an equivariant embedding of $S$.
\begin{enumerate}[label=\textup{(\arabic*)},leftmargin=*]
\item If $X$ is smooth, then $X$ admits a Frobenius splitting relative to
\[
(p-1)\left(\partial X+\sum_{i\in\I}X_{s_i}\right)
\]
that compatibly splits every boundary divisor and every stratum closure $X_u$.
\item Every equivariant embedding $X$ admits a Frobenius splitting compatible with every boundary divisor and every stratum closure $X_u$.
\end{enumerate}
\end{theorem}

\begin{proof}
We prove the two statements together. We follow the outline of \cite{Tan18}*{Proposition 2.1(i)}, using the distinguished sections of Section~\ref{sec:wonderful-splitting}. By \cite{Tan18}*{Theorem 2.1(i)}, there is a $G$-equivariant rational resolution
\[
q:\widetilde X\longrightarrow X
\]
such that $q$ is projective and $\widetilde X$ is a nonsingular quasi-projective toroidal embedding. Since $X$ is normal and $q$ is proper and birational, $q_*\CO_{\widetilde X}=\CO_X$. By \cite{BK05}*{Lemma 1.1.8}, a splitting of $\widetilde X$ descends to $X$, together with compatibility with the images of compatibly split subvarieties. It therefore suffices to construct the compatible splitting on $\widetilde X$; the relative assertion for smooth $X$ will follow at the end. Thus we may assume for the construction that $X$ is nonsingular, toroidal, and quasi-projective.

By \cite{Tim07}*{Section 29.2} or \cite{Kno91}, $X$ admits a nonsingular toroidal equivariant completion. A splitting of the completion restricts to $X$, as does compatibility with the relevant closed subschemes. We may therefore further assume that $X$ is complete. 

Lemma \ref{lem:embeddings-toroidal-map} gives a $G$-equivariant morphism $\pi:X\longrightarrow\X$.
The sections of Section~\ref{subsec:wonderful-distinguished-sections}
come from the restricting of
\[
\varphi_i\in E_{\overline{\omega_i}}
\cong H^0(\X,\CL_{-\overline{\omega_i}}).
\]
By Corollary~\ref{cor:si}, their zero schemes are
$Z(\varphi_i)=\X_{s_{\kappa(i)}}$.
Put $s=\prod_{i\in\I}\varphi_i$. Multiplication gives
$s\in E_{\overline\rho}$, and Proposition~\ref{prop:wc} identifies
\[
\CL_{-\overline\rho}\cong\widetilde\omega_{\X}^{-1}.
\]
If $\Y\subset\X$ is the closed orbit, put $\overline s=s|_{\Y}$.
By Proposition~\ref{prop:dv} and the proof of
Theorem~\ref{thm:WCsplit}, $\overline s^{p-1}$ defines the
closed-orbit splitting after a nonzero scalar normalization.
By Lemma \ref{lem:embeddings-toroidal-map}, the pullback
\[
\widetilde s=\pi^*s
\]
is a section of $\widetilde\omega_X^{-1}$. Let $s_{\partial X}$ be the canonical section of $\CO_X(\partial X)$ and set
\[
\sigma_X=s_{\partial X}\widetilde s\in H^0(X,\omega_X^{-1}).
\]
For every closed $G$-orbit $Y\subset X$, the restriction $\pi|_Y$ identifies $Y$ with $\Y$ and identifies $\widetilde s|_Y$ with $\overline s$. Hence $(\widetilde s|_Y)^{p-1}$ splits $Y$ after normalization by the proof of Theorem \ref{thm:WCsplit}.

The local expansion at a point of a closed $G$-orbit now shows that
\[
\sigma_X^{p-1}\in H^0(X,\omega_X^{1-p})
\]
defines a Frobenius splitting of $X$ after a nonzero scalar normalization. This is precisely the local argument in \cite{Tan18}*{Proposition 2.1(i)}; see also \cite{BK05}*{Theorem 1.3.8 and Exercise 1.3.E.4}. The factor $s_{\partial X}$ implies that every boundary divisor is compatibly split.

The explicit factorization of $s$ determines further compatibly split subschemes. Indeed, we have 
\[
\widetilde s=\prod_{i\in\I}\pi^*\varphi_i,
\]
and $Z(\sigma_X)$ is reduced and compatibly split by
\cite{BK05}*{Proposition~1.3.11}. On the open orbit, the support
of $Z(\pi^*\varphi_i)$ is $S_{s_{\kappa(i)}}$; this follows
from Corollary~\ref{cor:si} and the compatibility of the stratum
definitions with the adjoint quotient. Since $\kappa$ permutes
$\I$, every nonempty $X_{s_i}$ is an irreducible component of
$Z(\sigma_X)$ and is compatibly split.

We obtain compatibility with $X_u$ by Lemma~\ref{lem:prelim-strata-intersection}, standard combinatorics on Bruhat orders, and \cite{BK05}*{Proposition~1.2.1}.



Finally, when $X$ is smooth, the compatible splitting is relative to
\[
(p-1)\left(\partial X+\sum_{i\in\I}X_{s_i}\right)
\]
by \cite{BK05}*{Theorem 1.4.10}.
\end{proof}

\begin{remark}
For normal varieties, Frobenius splittings may be described on the regular locus using divisorial sheaves; see \cite{BK05}*{Remark~1.3.12}. Allowing effective Weil divisors, as in \cite{BK05}*{Remark~1.4.13}, Theorem~\ref{thm:main} can therefore be rephrased as the existence, for every equivariant embedding $X$, of a Frobenius splitting relative to $(p-1)\left(\partial X+\sum_{i\in\I}X_{s_i}\right)$, compatible with every boundary divisor and every $X_u$.
\end{remark}

\subsection{Strong $F$-regularity of $k[S_w]$}\label{subsec:bruhat-atlas}

Write $\mathcal B=G/B^+$. For $v,w\in W$, set
\[
\mathring{\mathcal B}_w=B^+\dot wB^+/B^+,
\qquad
\mathring{\mathcal B}^{v}=B^-\dot vB^+/B^+.
\]
Use the Richardson notation $\mathring{\mathcal R}_{v,w}$ and
$\mathcal R_{v,w}$ of Section~\ref{subsec:wonderful-distinguished-sections},
so $\mathring{\mathcal R}_{v,w}
=\mathring{\mathcal B}^{v}\cap\mathring{\mathcal B}_w$. 
In the formulas below, identify $S$ with $\iota(S)\subset G$.

Fix $v\in W$ and put
\[
N_v=\dot vU^-\dot v^{-1},\qquad N_v^\pm=N_v\cap U^\pm.
\]
Multiplication gives isomorphisms
\[
\begin{aligned}
N_v^+\times N_v^-&\xrightarrow{\sim}N_v,
& (g_+,g_-)&\longmapsto g_+g_-,\\
N_v^-\times N_v^+&\xrightarrow{\sim}N_v,
& (h_-,h_+)&\longmapsto h_-h_+.
\end{aligned}
\]
Define the factor projections by
\[
y=\sigma_{l,+}(y)\sigma_{r,-}(y)
 =\sigma_{l,-}(y)\sigma_{r,+}(y),\qquad y\in N_v,
\]
where $\sigma_{l,\pm},\sigma_{r,\pm}:N_v\to N_v^\pm$ are morphisms.
Here $l$ and $r$ indicate the left and right factors, and the dependence
on $v$ is suppressed. By \cite{KWY13}*{Lemma~2.2}, the pair of right
factor projections gives an isomorphism
\[
(\sigma_{r,+},\sigma_{r,-}):N_v\xrightarrow{\sim}N_v^+\times N_v^-.
\]
The open subsets $\dot vU^-B^+$, for $v\in W$, cover $G$;
see \cite{KWY13}*{proof of Theorem~1.1}.

Using these factor projections, we consider the isomorphism 
\begin{align*}
 \dot{v} U^- B^+ &\xrightarrow{\sim} \mathring{\mathcal{B}}^v \times B^+\dot{v} B^+,\\
 y \dot{v} b &\mapsto (\sigma_{r,-}(y)\dot{v} \cdot B^+, \sigma_{r,+}(y) \dot{v} b \dot{w}_\bullet^{-1} \theta(\sigma_{l,-}(y)) \dot{w}_\bullet).
\end{align*}

The factorization $y=\sigma_{l,-}(y)\sigma_{r,+}(y)$ gives
\[
\sigma_{r,+}(y) \dot{v} b \dot{w}_\bullet^{-1} \theta(\sigma_{l,-}(y)) \dot{w}_\bullet = \sigma_{l,-}(y)^{-1} \cdot y\dot{v} b \cdot \dot{w}_\bullet^{-1} \theta(\sigma_{l,-}(y)) \dot{w}_\bullet.
\]
Hence, by restriction, we obtain the isomorphisms 
\begin{equation}\label{eq:atlas1}
\mathring{S}_w \dot{w}_\bullet \cap \dot{v} U^-\dot{v}^{-1} \dot{v} B^+ \xrightarrow{\sim} \mathring{\mathcal{R}}_{v, ww_0} \times \mathring{S}_{vw_0} \dot{w}_\bullet,
\end{equation}
\begin{equation}\label{eq:atlas2}
{S}_w \dot{w}_\bullet \cap \dot{v} U^-\dot{v}^{-1} \dot{v} B^+ \xrightarrow{\sim} \bigcup_{w\le u\le w^\bullet}\mathring{\mathcal{R}}_{v,u w_0} \times \mathring{S}_{vw_0} \dot{w}_\bullet.
\end{equation}

\begin{remark}
The stratified isomorphisms \eqref{eq:atlas1} and \eqref{eq:atlas2} were suggested by Jeff York Ye. We thank him for helpful discussions.
\end{remark}

\begin{prop}\label{prop:RichardsonFregular}Assume $\operatorname{char}k=p>2$. 
Richardson varieties are globally $F$-regular. 
\end{prop}

\begin{proof}
The case $v=w$ is immediate. Write
$\partial\mathcal{R}_{v,w}=\bigcup_iE_i$ as the union of its irreducible
components. The construction in \cite{KS}*{equation~(4) and
Section~4.5}, valid also in positive characteristic, gives an ample
effective Cartier divisor $D=\sum_i b_iE_i$ with $b_i>0$.
The Frobenius splitting in \cite{KS}*{Proposition~5.3}, together with
\cite{BT06}*{Lemma~3.1}, shows that $\mathcal{R}_{v,w}$ is stably
Frobenius split along $D$.
The complement $\mathring{\mathcal{R}}_{v,w}$ is smooth, hence strongly $F$-regular by \cite{HH89}*{Theorem~3.1(c)}.
The proposition follows from Theorem~\ref{thm:gloFre}(2).
\end{proof}

\begin{prop}\label{prop:DBFregular}
Assume $\operatorname{char}k=p>2$. For $w\le w^\bullet$, the coordinate ring $k[S_w]$ is strongly $F$-regular.
\end{prop}

\begin{proof}
Recall that strong $F$-regularity is local by
\cite{HH89}*{Theorem~3.1(a)}. On each nonempty chart, the stratum $\mathring{S}_{vw_0}$ is
smooth by Lemma~\ref{lem:prelim-strata-intersection}, hence strongly
$F$-regular by \cite{HH89}*{Theorem~3.1(c)}.
Moreover,
\[
\bigcup_{w\leq u\leq w^\bullet}\mathring{\mathcal{R}}_{v,uw_0}
=\mathcal{R}_{v,ww_0}\cap\mathring{\mathcal{B}}^{v}
\]
is an open subset of the Richardson variety $\mathcal{R}_{v,ww_0}$,
so it is strongly $F$-regular by
Proposition~\ref{prop:RichardsonFregular} and
\cite{BT06}*{Section~3}.
By the isomorphism \eqref{eq:atlas2}, each chart of $S_w\dot w_\bullet$
is a product of these two strongly $F$-regular varieties,
and is therefore strongly $F$-regular.
These charts cover $S_w\dot w_\bullet$.
By locality, $k[S_w]$ is strongly $F$-regular.
\end{proof}
 
\subsection{Global \(F\)-regularity of stratum closures}\label{subsec:global-f-regularity}

\begin{proposition}\label{prop:AmpleDiv}
Assume that $\operatorname{char}k\neq2$. Let $X$ be a smooth projective
toroidal embedding of the symmetric space $S$, with boundary divisors
$D_1,\dots,D_n$. Then there exist positive integers $a_1,\dots,a_n$ such
that $\sum_{i=1}^n a_iD_i$ is an ample effective Cartier divisor.
\end{proposition}

\begin{proof}
Let $\X$ be the wonderful compactification of $S_{\mathrm{ad}}$.
By Lemma~\ref{lem:embeddings-toroidal-map}, the canonical map
$S\to S_{\mathrm{ad}}$ extends to a $G$-equivariant morphism
\[
\pi:X\longrightarrow\X.
\]
The local structure results used below apply to symmetric spaces in
characteristic different from $2$; see \cite{DCS99}*{Proposition~3.8 and
Theorem~3.9} and \cite{Tan18}*{Section~3, pp.~502--503}.

Choose the standard affine chart $\X_0=\X^-\subseteq\X$ of
Section~\ref{subsec:wonderful-closed-orbit}. The compatible
local structure theorems give
\begin{equation}\label{eq:ample-boundary-chart}
\X_0\simeq U\times\BA^{r_{\mathrm{ad}}},
\qquad
X_0:=\pi^{-1}(\X_0)\simeq U\times Z,
\end{equation}
where $U=U_P^-$, $r_{\mathrm{ad}}$ is the rank of
$S_{\mathrm{ad}}$, and $Z$ is a toric variety with dense torus
$\breve T$ whose fan is the fan of $X$. The boundary divisors of $X$
correspond bijectively to the torus-invariant prime divisors
$Z_1,\dots,Z_n$ of $Z$, with
\begin{equation}\label{eq:ample-boundary-components}
D_i\cap X_0=U\times Z_i.
\end{equation}
For the equality of fans and the boundary correspondence, see
\cite{Tan18}*{Lemma~1.2(i)}; the inverse-image description in
\eqref{eq:ample-boundary-chart} is used in
\cite{Tan18}*{proof of Theorem~2.1(i), p.~501}.
The $G$-translates of $\X_0$ cover $\X$.

Since $Z$ is closed in the open subvariety $X_0$ of the projective
variety $X$, it is quasiprojective. A quasiprojective toric variety
has an effective ample Cartier divisor supported on its toric boundary:
this follows from the invariant-divisor description of its Picard group
and the character decomposition of global sections in
\cite{CLS11}*{Theorem~4.2.1 and Proposition~4.3.3}, applied to a very
ample line bundle. These descriptions are valid over any algebraically
closed field; the same consequence is used in
\cite{Tan18}*{proof of Theorem~2.1(i), p.~501}.
Choose such a divisor
\[
E_Z=\sum_{i=1}^n b_iZ_i,\qquad b_i\in\BZ_{\geq0},
\]
and set $E=\sum_i b_iD_i$. Since $X$ is smooth, $E$ is Cartier.
By \eqref{eq:ample-boundary-components},
\[
\CO_X(E)|_{X_0}\simeq\operatorname{pr}_Z^*\CO_Z(E_Z).
\]
This sheaf is ample, since $\operatorname{pr}_Z:U\times Z\to Z$ is
affine and pullback by an affine morphism preserves ampleness
\cite{Stacks}*{Tag~0892(2)}. Since $E$ is $G$-stable and
$\pi^{-1}(g\X_0)=gX_0$, its restriction to $\pi^{-1}(g\X_0)$ is
ample for every $g\in G(k)$. Relative ampleness is local on the base
\cite{Stacks}*{Tag~01VJ}; hence $E$ (on $X$) is $\pi$-ample.

We next choose an ample effective Cartier divisor $H$ on $\X$ whose
support is $\partial\X$. Use the projective realization
\[
\X=\overline{G_{\mathrm{ad}}\cdot[h]}
\subseteq\BP\bigl(\operatorname{Hom}_k({}^{\theta_{\mathrm{ad}}}V,V)\bigr)
\]
of \cite{DCS99}*{construction preceding Lemma~3.4 and Theorem~3.9}.
The open orbit consists of invertible maps, while every boundary point
represents a singular map; this is stated in
\cite{DCS99}*{proof of Theorem~3.9(ii), p.~287}.
Here ${}^{\theta_{\mathrm{ad}}}V$ is the twist of $V$ by
$\theta_{\mathrm{ad}}$, and $h$ is the canonical linear isomorphism.
Thus the determinant restricts to a nonzero section
$\delta\in H^0(\X,\CO_\X(d))$, where $d=\dim V$, and
\[
H=\operatorname{div}_{\X}(\delta),
\qquad
\operatorname{Supp}(H)=\partial\X,
\qquad
\CO_\X(H)\simeq\CO_\X(d).
\]
In particular, $H$ is an ample effective Cartier divisor.

The pullback $\pi^*H$ is an effective Cartier divisor, since its defining
section is nonzero on the dense open orbit $S$. Its support is contained
in $\partial X$, because $\pi(S)\subseteq S_{\mathrm{ad}}$.
Since $E$ is $\pi$-ample and $H$ is ample, there exists an integer
$N\geq0$ such that
\[
A=E+N\pi^*H
\]
is ample by \cite{Stacks}*{Tag~0892(1)}. Hence
$A=\sum_i c_iD_i$ with $c_i\in\BZ_{\geq0}$.

It remains to make every coefficient of $A$ positive. Put $\Delta=\sum_iD_i$,
which is Cartier because $X$ is smooth. For $m\gg0$, the sheaf
$\CO_X(\Delta+(m-1)A)$ is globally generated
\cite{Stacks}*{Tag~01Q3(7)}. Tensoring with the ample sheaf
$\CO_X(A)$ shows that $\CO_X(\Delta+mA)$ is ample
\cite{Stacks}*{Tag~0890}. Therefore
\[
D=\Delta+mA=\sum_{i=1}^n(1+mc_i)D_i
\]
is an ample effective Cartier divisor with all coefficients strictly
positive, as required.
\end{proof}

\begin{theorem}\label{thm:global-f-regular-strata}
Assume that $\operatorname{char}k=p>2$. Let $X$ be a projective
equivariant embedding of $S$. Then $X_u$ is globally $F$-regular for
every $u\leq w^\bullet$.
\end{theorem}

\begin{proof}
Suppose first that $X$ is smooth and toroidal. By
Proposition~\ref{prop:AmpleDiv}, there is an ample effective Cartier
divisor
\[
D=\sum_{i=1}^n a_iD_i,\qquad a_i>0,
\]
with $\operatorname{Supp}(D)=\partial X$. Let $\varphi$ be the
splitting constructed in Theorem~\ref{thm:main}. It splits $X$
relative to $(p-1)\partial X$ and compatibly splits every $X_u$.
Iterating this relative splitting gives a stable splitting along
$(p^e-1)\partial X$. Choose $e$ such that $a_i\leq p^e-1$ for all
$i$, and restrict along
\[
\CO_X(D)\hookrightarrow\CO_X((p^e-1)\partial X).
\]
We obtain a stable splitting along $D$ whose induced Frobenius
splitting still compatibly splits every $X_u$; see
\cite{BT06}*{Lemma~3.1(4)--(5)} and
\cite{BK05}*{proof of Lemma~1.4.11}.

Since $X_u$ meets $S$, it is not contained in $\operatorname{Supp}(D)$.
Consequently, $D|_{X_u}$ is an ample effective Cartier divisor, and
the stable splitting descends to a stable splitting of $X_u$ along
$D|_{X_u}$ by \cite{BK05}*{Proposition~1.4.3}, applied to the
iterated Frobenius. Moreover,
\[
X_u\setminus\operatorname{Supp}(D|_{X_u})=X_u\cap S=S_u,
\]
which is strongly $F$-regular by 
Proposition~\ref{prop:DBFregular}. Theorem~\ref{thm:gloFre}(2)
therefore implies that $X_u$ is globally $F$-regular.

For an arbitrary projective embedding $X$, choose a projective
$G$-equivariant resolution
\[
q:\widetilde X\longrightarrow X
\]
with $\widetilde X$ smooth and toroidal, as in
\cite{Tan18}*{Theorem~2.1(i)}. Then $\widetilde X$ is projective and
$q_*\CO_{\widetilde X}=\CO_X$, since $X$ is normal. Let
$\widetilde X_u$ be the closure of $S_u$ in $\widetilde X$, and write
$q_u:\widetilde X_u\to X_u$ for the restriction of $q$. The preceding
argument gives a stable splitting of $\widetilde X$ along an ample
effective Cartier divisor supported on $\partial\widetilde X$,
compatible with $\widetilde X_u$. Since this divisor does not contain
$\widetilde X_u$, \cite{BT06}*{Lemma~3.2} gives
\[
(q_u)_*\CO_{\widetilde X_u}=\CO_{X_u}.
\]
The variety $\widetilde X_u$ is globally $F$-regular by the smooth
toroidal case. Hence $X_u$ is globally $F$-regular by
\cite{BT06}*{Corollary~3.2}.
\end{proof}

\begin{corollary}\label{cor:strong-f-regular-strata}
Assume that $k$ is algebraically closed of characteristic $p>2$.
Let $X$ be an equivariant embedding of $S$. Then $X_u$ is strongly
$F$-regular for every $u\leq w^\bullet$.
\end{corollary}

\begin{proof}
By Sumihiro's theorem \cite{Bri17}*{Theorem~1 and Corollary~2.14},
$X$ has a finite $G$-stable open cover by equivariant embeddings
which are open subvarieties of projective equivariant embeddings.
As in \cite{BT06}*{proof of Corollary~4.2}, the assertion follows
from Theorem~\ref{thm:global-f-regular-strata} and the locality of
strong $F$-regularity \cite{HH89}*{Theorem~3.1(a)}.
\end{proof}

\subsection{The characteristic-zero case}\label{subsec:characteristic-zero}

Assume that $k$ is algebraically closed of characteristic zero.
We use strongly and globally $F$-regular type as in
\cite{BT06}*{Section~4.4}; see also \cite{Smi00}*{Definition~5.1}.

\begin{theorem}\label{thm:characteristic-zero-f-regular}
Let $X$ be an equivariant embedding of $S$. Then $X_u$ is of
strongly $F$-regular type for every $u\leq w^\bullet$.
If $X$ is projective, then each $X_u$ is of globally $F$-regular type.
\end{theorem}

\begin{proof}
In the work \cite{BS24}*{Theorems~4.4 and~4.8}, we construct an integral form of $k[S]$ whose base change
to every algebraically closed field of characteristic different
from $2$ is the coordinate ring of the corresponding symmetric
space.
This provides the integral homogeneous space needed for the
reduction argument used
in \cite{HT07}*{proof of Theorem~9.1}.
We omit the remaining details.
\end{proof}

\begin{corollary}\label{cor:characteristic-zero-singularities}
Let $X$ be an equivariant embedding of $S$. Then $X_u$ is normal,
Cohen--Macaulay, and has rational singularities for every
$u\leq w^\bullet$.
\end{corollary}

\begin{proof}
By Theorem~\ref{thm:characteristic-zero-f-regular}, each $X_u$ is of
strongly $F$-regular type, hence of $F$-rational type by definition; see \cite{Smi97} for the definition of $F$-rationality. 
The assertion
follows from \cite{Smi97}*{Theorem~4.3}.
\end{proof}

\end{document}